\documentclass[]{elsarticle}

\usepackage{amssymb,amsmath,amsthm,array,eufrak}
\usepackage{dsfont,graphicx}
\usepackage{pst-plot,pstricks,multido}

\journal{a journal}

\newtheorem{thm}{Theorem}[section]
\newtheorem{lem}[thm]{Lemma}
\newtheorem{cor}[thm]{Corollary}

\newtheorem{prop}[thm]{Proposition}
\theoremstyle{definition}
\newtheorem{defn}{Definition}

\newtheorem{exam}{Example}
\newtheorem*{remk}{Remark}

\newcommand{\R}{\mathbb{R}}
\newcommand{\norm}[1]{\left\Vert#1\right\Vert}
\newcommand{\abs}[1]{\left\vert#1\right\vert}
\newcommand{\set}[1]{\left\{#1\right\}}
\newcommand{\paren}[1]{\left(#1\right)}
\newcommand{\brac}[1]{\left[#1\right]}

\DeclareMathOperator{\spanof}{span}
\DeclareMathOperator{\setdiff}{\triangle}

\newcommand{\commentout}[1]{}
\newcommand{\draftnote}[1]{}

\DeclareMathOperator{\graph}{graph}

\DeclareMathOperator{\Inv}{Inv}
\begin{document}


\begin{frontmatter}
\title{Shuffles of copulas and a new measure of dependence} 
\author[chula]{P.~Ruankong}
\ead{ruankongpol@gmail.com}

\author[chula,cem]{T.~Santiwipanont\corref{cor}\corref{thanks}}
\ead{tippawan.s@chula.ac.th}

\author[chula,cem]{S.~Sumetkijakan\corref{thanks}}
\ead{songkiat.s@chula.ac.th}

\address[chula]{Department of Mathematics and Computer Science, Faculty of Science, Chulalongkorn University, \mbox{Phyathai Road}, Patumwan, Bangkok 10330, Thailand}

\address[cem]{Centre of Excellence in Mathematics, CHE, Si Ayutthaya Rd., Bangkok 10400, Thailand}

\cortext[cor]{Corresponding author}
\cortext[thanks]{The second- and third-named authors are partially supported by the Centre of Excellence in Mathematics, the Commission on Higher Education, Thailand.}

\begin{abstract}
Using a characterization of Mutual Complete Dependence copulas, we show that, with respect to the Sobolev norm, the MCD copulas can be approximated arbitrarily closed by shuffles of Min. This result is then used to obtain a characterization of generalized shuffles of copulas introduced by Durante, Sarkoci and Sempi in terms of MCD copulas and the $*$-product discovered by Darsow, Nguyen and Olsen.    
Since shuffles of a copula is the copula of the corresponding shuffles of the two continuous random variables, we define a new norm which is invariant under shuffling.  This norm gives rise to a new measure of dependence which shares many properties with the maximal correlation coefficient, the only measure of dependence that satisfies all of R\'{e}nyi's postulates.

%
\end{abstract}

\end{frontmatter}

\section{Introduction}
Since the copula of two continuous random variables is scale-invariant, 
copulas are regarded as the functions that capture dependence structure between random variables.  
%
For many purposes, independence and monotone dependence have so far been considered two opposite extremes of dependence structure. 
%
%
However, monotone dependence is just a special kind of dependence between two random variables.  More general complete dependence happens when functional relationship between continuous random variables are piecewise monotonic, which corresponds to their copula being a shuffle of Min.  See \cite{
siburg2008spc,Siburg2009mmc}.   Mikusinski et al.~\cite{mikusinski1992sm,mikusinski1991probabilistic} showed that shuffles of Min is dense  in the class of all copulas with respect to the uniform norm.  This surprising fact urged the discovery of the (modified) Sobolev norm by Siburg and Stoimenov \cite{Siburg2009mmc} which is based on the $*$-operation introduced by Darsow et al.~\cite{darsow1992copulas,darsow1995nc,olsen1996copulas}.  They \cite{darsow1992copulas,darsow1995nc,olsen1996copulas,siburg2008spc,Siburg2009mmc} showed that continuous random variables $X$ and $Y$ are mutually completely dependent, i.e.~their functional relationship is any Borel measurable bijection, if and only if their copula has unit Sobolev-norm.  
Darsow et al.~\cite{darsow1992copulas,darsow1995nc,olsen1996copulas} 
showed that for a real stochastic processes $\{X_t\}$, the validity of the Chapman-Kolmogorov equations is equivalent to the validity of the equations $C_{st} = C_{su}*C_{ut}$ for all $s<u<t$, where $C_{st}$ denotes the copula of $X_s$ and $X_t$.  It is then natural to investigate how dependence levels of $A$ and $B$ are related to that of $A*B$. Aside from $\Pi$, $M$ and $W$, the easiest case is when $A$ and $B$ are mutual complete dependence copulas. 
In light of our result on denseness of shuffles of Min in the MCD copulas, we shall show 
that 
if $\norm{A} = \norm{B} = 1$ then $\norm{A*B} = 1$.  
Now, if $\norm{A}=1$ and $C$ is a copula then we prove that $A*C$ coincides with a generalized shuffle of $C$ in the sense of Durante et al.~\cite{durante2008sc}.   We also give similar characterizations of shuffles of $C$ and generalized shuffles of Min.  
These characterizations have advantages of simplicity in calculations because it avoids using induced measures.  Then we use this relationship to obtain a simple proof of a characterization of copulas whose orbit is singleton (Theorem 10 in \cite{durante2008sc}).  
Note that there are many examples where shuffles of $C$, i.e.~$A*C$ or $C*A$, do not  have the same Sobolev norm as $C$.  However, we show that multiplication by unit norm copulas preserves independence, complete dependence and mutual complete dependence.

Since left- and right-multiplying a copula $C=C_{X,Y}$ by unit norm copulas amount to ``shuffling'' or ``permuting'' $X$ and $Y$ respectively, we introduce a new norm, called the $*$-norm, which is invariant under multiplication by a unit norm copula.  Mutual complete dependence copulas still has $*$-norm one.  This invariant property implies that complete dependence copulas also possess unit $*$-norm.  Based on the $*$-norm, a new measure of dependence is defined in the same spirit as the definition by Siburg et al.~\cite{Siburg2009mmc}.  It turns out that this new measure of dependence satisfies most of the seven postulates proposed by R\'{e}nyi \cite{renyi1959}.  The only known measure of dependence that satisfies all R\'{e}nyi's postulates is the maximal correlation coefficient.

This manuscript is structured as follows. We shall summarize related basic properties of copulas, the binary operator $*$ and the Sobolev norm 
in Section 2.  Then we obtain a characterization of copulas with unit Sobolev norm which implies that the $*$-product of MCD copulas is a MCD copula in Section 3. 
Section 4 contains our characterizations of generalized shuffles of Min and (generalized) shuffles of copulas in the sense of Durante et al.~in terms of the $*$-product.  We then show that shuffling a copula preserves independence, complete dependence and mutual complete dependence.  In Section 5, a new norm is introduced and its properties are proved.  And in Section 6, we define a new measure of dependence and verify that it satisfies most of R\'{e}nyi's postulates.

\section{Basics of copulas}
A \emph{bivariate copula} is defined to be a joint distribution function of two random variables with uniform distribution on $[0, 1].$ Since such a joint distribution is uniquely determined by its restriction on $[0, 1]^2$ one can also define a copula as a function $C \colon [0, 1]^2 \to [0, 1]$ satisfying the following properties
\begin{equation}
C(u, 0) = 0 = C(0, v),\; C(u, 1) = u,\; C(1, v) = v,\quad\text{ and}
\end{equation}
\begin{equation}
C(u, v)-C(u, y)-C(x, v)+C(x, y) \geq 0
\end{equation}
for all $(u, v) \in [0, 1]^2$ and $(x, y) \in [0, 1]^2$ such that $x \leq u, y \leq v.$ Note that the two definitions are equivalent.  Every copula $C$ induces a measure $\mu_C$ on $[0,1]^2$ by $$\mu_C([x,u]\times[y,v]) = C(u, v)-C(u, y)-C(x, v)+C(x, y).$$  The induced measure $\mu_C$ is doubly stochastic in the sense that for every Borel set $B$, $\mu_C([0,1]\times B) = m(B) = \mu_C(B\times[0,1])$ where $m$ is Lebesgue measure on $\R$.
Important copulas include the Fr\'echet-Hoeffding upper and lower bounds
$$M(x, y) = \min(x, y), \quad W(x,y) = \max(x+y-1,0)$$
and the product, or independent, copula $\Pi(x, y) = xy$.
A fundamental property is that $M$ is a copula of $X$ and $Y$ \draftnote{continuous?} if and only if $Y$ is almost surely an increasing bijective function of $X$. If $X$ and $Y$ are uniformly distributed on $[0,1]$ then that bijection is the identity map on $[0,1]$. Its graph, the main diagonal, is the support of the induced measure $\mu_M$, also called the support of $M$.
At the other extreme, the minimum copula $W(x,y) = \max(x+y-1,0)$ corresponds to random variables being monotone decreasing function of each other.  \draftnote{continuous?}

Listed below are some basic properties of any copula $C$, some of which shall be used frequently in the manuscript.
\begin{enumerate}
  \item $W(x, y) \leq C(x, y) \leq M(x, y)$ for all $x, y \in [0, 1].$
  \item $\abs{C(u,v) - C(x,y)} \leq |u-x|+|v-y|$ $\forall (u,v),(x,y)\in[0,1]^2$ and hence $C$ is uniformly continuous.
  \item $\partial_1C$ and $\partial_2C$ exist almost everywhere on $[0,1]^2$.
  \item For a.e.~$x\in[0,1]$, $\partial_1C(x,\cdot)$ is nondecreasing in the domain where it exists and similar statement holds for $\partial_2(\cdot,y)$.
\end{enumerate}
Perhaps, the most important property of copulas is given by the Sklar's theorem which states that to every joint distribution function $H$ of continuous random variables $X$ and $Y$ with marginal distributions $F$ and $G,$ respectively, there corresponds a unique copula $C$, called the \emph{copula of $X$ and $Y$} for which
\[H(x, y) = C(F(x), G(y))\] for all $x, y \in \R.$ This means that the copula of $(X, Y)$ captures all dependence structure of the two random variables. $X$ and $Y$ are said to be \emph{mutually completely dependent} if there exists an invertible Borel measurable function $f$ such that $P(Y = f(X)) = 1.$ Shuffles of Min were introduced by Mikusinski et al.~\cite{mikusinski1992sm} as examples of copulas of mutually completely dependent random variables. By definition, a shuffle of Min is constructed by shuffling (permuting) the support of the Min copula $M$ on $n$ vertical strips subdivided by a partition $0=a_0<a_1<\cdots<a_n=1$.
It is shown 
\cite[Theorems 2.1 \& 2.2]{mikusinski1992sm} 
that the copula of $X$ and $Y$ is a shuffle of Min if and only if there exists an invertible Borel measurable function $f$ with finitely many discontinuity points such that $P(Y=f(X))=1$. In  \cite{mikusinski1992sm}, such an $f$ is called strongly piecewise monotone function. 

Following \cite{darsow1992copulas,darsow1995nc}, the binary operation $*$ on the set $\mathcal{C}_2$ of all bivariate copulas is defined as
\[C*D(x,y) = \int_0^1 \partial_2C(x,t)\partial_1D(t,y)\,dt\quad\text{for } x,y\in[0,1]\]
and the \emph{Sobolev norm} of a copula $C$ is defined  by \[\norm{C}^2 = \int_0^1\int_0^1\abs{\nabla C(x,y)}^2\,dx\,dy = \int_0^1\int_0^1\paren{\partial_1C^2(x,y)+\partial_2C^2(x,y)}\,dx\,dy.\] 
It is well-known that $(\mathcal{C}_2,*)$ is a monoid with null element $\Pi$ and identity $M$. So a copula $C$ is called \emph{left invertible} (\emph{right invertible}) if there is a copula $D$ for which $D*C = M$ ($C*D=M$).
It was shown in 
\cite[Theorem 7.6]{darsow1992copulas} and \cite[Theorem 4.2]{darsow1995nc} that the $*$-product on $\mathcal{C}_2$ is jointly continuous with respect to the Sobolev norm but not with respect to the uniform norm. Moreover, they \cite{darsow1992copulas,darsow1995nc,siburg2008spc} gave a statistical interpretation of the Sobolev norm of a copula.
\begin{thm}[{\cite[Theorems 4.1-4.3]{Siburg2009mmc}}]  
\label{thm:normC}
  Let $C$ be a bivariate copula of continuous random variables $X$ and $Y$. Then 
1.) $\displaystyle\frac{2}{3} \leq \norm{C}^2 \leq 1$; 
2.) $\displaystyle\norm{C}^2 = \frac{2}{3}$ if and only if $C = \Pi$; and 
3.) The following are equivalent.
    \begin{enumerate}
    \renewcommand{\theenumi}{\alph{enumi}}%
      \item $\norm{C}=1$.
      \item $C$ is invertible with respect to $*$.
      \item \label{thm:unitnorm3} For each $x,y\in[0,1]$, $\partial_1C(\cdot,y), \partial_2C(x,\cdot) \in \set{0,1}$ a.e.
      \item There exists a Borel measurable bijection $h$ such that $Y=h(X)$ a.e.
    \end{enumerate}
\end{thm}
It follows readily that all shuffles of Min have norm one.

\section{Copulas with unit Sobolev norm}
 Let $C$ be a copula with unit Sobolev norm. Then 
  $\partial_1 C(x, y)$ and $\partial_2 C(x, y)$ take values $0$ or $1$ almost everywhere. See, for example, Theorem 7.1 in \cite{darsow1992copulas} and Theorem 4.2 in \cite{Siburg2009mmc}.   Let us recall from \cite[Theorem 2.2.7]{nelsen2006ic}  
 that, for a.e.~$x \in [0, 1], \partial_1 C(x, y)$ is a nondecreasing function of $y \in [0, 1].$ Similar statement holds also for $\partial_2 C(x, y).$
 So for a.e.~$x \in [0, 1],$ there is $f(x) \in [0, 1]$ such that for almost every $y$, $\partial_1 C(x, y) = 1$ if $y > f(x)$ and $\partial_1 C(x, y) = 0$ if $y < f(x).$ 
 ($f(x) \equiv \sup\{ y\colon \partial_1 C(x, y)=0 \}$) Denote the set of such $x$'s by $\tilde I$ so that $m(\tilde I)=1$.  And for every $x\in \tilde I$, by redefining $\partial_1C(x,y)$ on a set of measure zero, we may assume that $\partial_1C(x,y)$ is defined and nondecreasing for all $y\in[0,1]$. To show that $f$ is measurable, let $\alpha \in [0, 1],$ and observe that since $\partial_1C(x,y)$ is increasing in $y$
\begin{align*}
  \{ x\in \tilde I \colon f(x) > \alpha \}
  &= \{ x\in \tilde I \colon \exists y > \alpha, \partial_1 C(x, y) = 0 \} \\
  &= \bigcup_{n=1}^{\infty} \{ x\in \tilde I \colon \partial_1 C\paren{x, \alpha + \frac{1}{n}} = 0 \}
\end{align*}
which is measurable because each $\partial_1 C(\cdot, \alpha + \frac{1}{n})$ is measurable.
In exactly the same fashion, there exists a measurable function $g\colon [0, 1] \to [0, 1]$ for which
$$\partial_2 C(x, y) = \begin{cases}
  1 &\text{if } x > g(y) \\
  0 &\text{if } x < g(y)
\end{cases} \;
\text{ for a.e.~} y, \text{ a.e.~} x.$$
Let us recall the definition that a measurable function $\phi\colon[0,1]\to[0,1]$ is said to be \emph{measure-preserving} if $m(\phi^{-1}(B)) = m(B)$ for any Lebesgue measurable set $B\subseteq[0,1]$. 
\begin{thm}\label{thm:suppcopulanormone}
  Suppose $C$ is a copula with unit Sobolev norm. Then there exists a unique invertible Borel measurable function $f\colon [0, 1] \to [0, 1]$ such that $f$ is measure-preserving and for almost every $(x,y)$ in $[0,1]^2$
  \begin{equation}\label{eq:f01}
   \partial_1 C(x, y) = \begin{cases}
    1 \; &\text{ if } y > f(x) \\
    0 \; &\text{ if } y < f(x)
  \end{cases} \;\text{ and }\;\;
  \partial_2 C(x, y) = \begin{cases}
    1 \; &\text{ if } x > f^{-1}(y) \\
    0 \; &\text{ if } x < f^{-1}(y).
  \end{cases}
  \end{equation}
  Furthermore, if $f$ is continuous on an interval $I$ then it is differentiable on $I$ with constant derivative equal to either $1$ or $-1$. 
\end{thm} 
\begin{remk}
  During the preparation of this manuscript, we have come across similar results such as 
  Proposition 1 in \cite{deAmo2010} and Theorem 2.4 and Corollary 2.4.1 in \cite{darsow2010}.
\end{remk}
\begin{proof}
We first claim that $f$ and $g$ defined above are inverses of each other in the sense that $f \circ g$ and $g \circ f$ are identity on $[0, 1]$ a.e., \draftnote{I don't think this can be proved via Thm 2.1(3)}
i.e.~$\set{x\colon x = g(f(x))}$ and $\set{y\colon y = f(g(y))}$ both have measure $1$. This is equivalent to saying that $y=f(x)$ if and only if $x=g(y)$ for a.e.~$(x,y)\in[0,1]^2$. 
Indeed, observe that
for any open interval $B\subset [0,1]$, $ f(x)\in B$ if and only if $\partial_1C(x,B) = \set{0,1}$.
Now let $A=(a_1,a_2)$ and $B=(b_1,b_2)$ be open intervals in $[0,1]$ for which $A\times B$ does not intersect the graph $y=f(x)$, i.e.~$m\paren{\set{x\in A\colon f(x)\in B}}=0$, hence $\partial_1C(x,y)$ is independent of $y\in B$ for a.e.~$x\in A$.
So $\partial_1 C(x, y) = \delta(x) \equiv 0 \text{ or } 1$ for a.e.~$x\in A$ and all $y\in B$.  Then for $ (x, y)$ in $A\times B $
  \begin{equation*}
    C(x, y) = \int_{0}^{x} \partial_1 C(t, y) \,dt
    = \int_{0}^{a_1} \partial_1 C(t, y) \,dt + \int_{a_1}^{x} \delta(t) \,dt 
    = C(a_1, y) + \int_{a_1}^{x} \delta(t) \,dt
  \end{equation*}
and so $\partial_2 C(x, y) = \partial_2 C(a_1, y)$ is independent of $x\in A$ which implies that $A\times B$ does not intersect the graph $x=g(y)$. The converse can be shown by a 
similar argument.  Since the graph of a Borel function  \draftnote{Really?} is a Borel subset of $[0,1]^2$, $y=f(x)$ and $x=g(y)$ give the same graph. And the claim follows.


Let $\mu_C$ denote the doubly stochastic measure associated with $C$.  A straightforward verification gives 
\begin{align} \label{eq:mp}
  \mu_C(A\times B) = m(A\cap f^{-1}(B))
\end{align}
for all Borel rectangles $A\times B$, which implies by a standard measure-theoretic technique that 
\eqref{eq:mp} holds for all Borel sets $A,B \subseteq [0,1]$. So $f$ is measure-preserving since it is equivalent to the validity of \eqref{eq:mp} for all Borel sets $A$ and $B$.


Lastly, we prove that if $f$ is continuous on an open interval $I=(a,b)$ then it is differentiable with $f'$ being constant and equal to $\pm{1}$. 
Since $f$ is continuous and one-to-one on $I$, it has to be strictly monotonic on $I$.
Let us consider the case where $f$ is strictly increasing on $I$. This implies that $g=f^{-1}$ is strictly increasing on the interval $f(I)$ and that $y > f(x)$ if and only if $x < g(y)$ for $x\in I$.
  For $x\in [a_0, b_0]\subset (a,b),$
  \begin{align*}
    C(x, y)
      &= \int_{0}^{a_0} \partial_1 C(t, y) \,dt + \int_{a_0}^{x} \partial_1 C(t, y) \,dt
      = C(a_0, y) + \int_{a_0}^x \chi_{\set{t\colon y>f(t)}}\,dt\\
      &= C(a_0, y) + \int_{a_0}^x \chi_{[0,g(y))}\,dt
      = C(a_0, y) + \begin{cases}
         x - a_0  &\text{if } x < g(y), \\
         g(y) - a_0  &\text{if } x > g(y).
        \end{cases}
  \end{align*}
  Since $C(x,y)$ and $C(a_0,y)$ are differentiable with respect to $y$ almost everywhere, we have for a.e.~$y$, $$\partial_2 C(x, y) = \partial_2 C(a_0, y) + \begin{cases}
    0 \; &\text{ if } x < g(y) \\
    g'(y) \; &\text{ if } x > g(y).
  \end{cases}$$
  As $g'(y) > 0$ 
  and $\partial_2 C(x, y)$ and $\partial_2 C(a_0, y)$ are equal to $0$ or $1$, $g'(y) = 1$ and hence $f'(x)=1$ for all $x\in I$.
Similarly, if $f$ is strictly decreasing on $(a,b)$ then $f'= -1$ a.e.~on $(a,b)$.
\end{proof}

A natural question is then to investigate the set on which an invertible measure-preserving function $f$ is continuous. Unfortunately, the support of a unit norm copula may be the graph of a function which is discontinuous on a dense subset of $[0,1]$, and hence there is no interval on which it is continuous.
\begin{exam}
  Define a sequence of shuffles of Min $\set{S_n}$ by letting $S_0$ be the comonotonic copula supported on the main diagonal. $S_1$ is defined so that it shares the same support with $S_0$ in $[0,\frac{1}{2}]\times[0,1]$ and its support in the other half $F_0\times[0,1]=[\frac{1}{2},1]\times[0,1]$ is that of $S_0$ flipped horizontally.
  $S_2$ is then obtained from $S_1$ by flipping the support in each stripe of the set $F_1\times[0,1]$ 
   where $F_1=[\frac{1}{2^2},\frac{1}{2}]\cup\paren{[\frac{1}{2^2},\frac{1}{2}]+\frac{1}{2}}$. For general $n\geq 1$, we define $F_n = \frac{1}{2}F_{n-1}\cup\paren{\frac{1}{2}F_{n-1}+\frac{1}{2}}$ and let the shuffle of Min $S_n$ be obtained from $S_{n-1}$ by flipping the support in each stripe of $F_n$ horizontally. To sum up, each shuffle of Min $S_n$ is supported on the graph $\set{(x,y)\colon y=f_n(x)}$ where $f_n$ is constructed according to the above iterative procedure, starting from $f_0(x)=x$ and $f_1(x) = x\,\chi_{[0,\frac{1}{2})} + \paren{\frac{3}{2}-x}\chi_{[\frac{1}{2},1]}$.
   The first few $S_n$'s are illustrated in Figure \ref{fig:S_n}. 
  \begin{figure}
  \label{fig:S_n}
  \begin{tabular}{@{}m{.35\textwidth}@{\hspace*{-.4cm}}m{.8\textwidth}@{}}
    \includegraphics[width=0.26\textwidth]{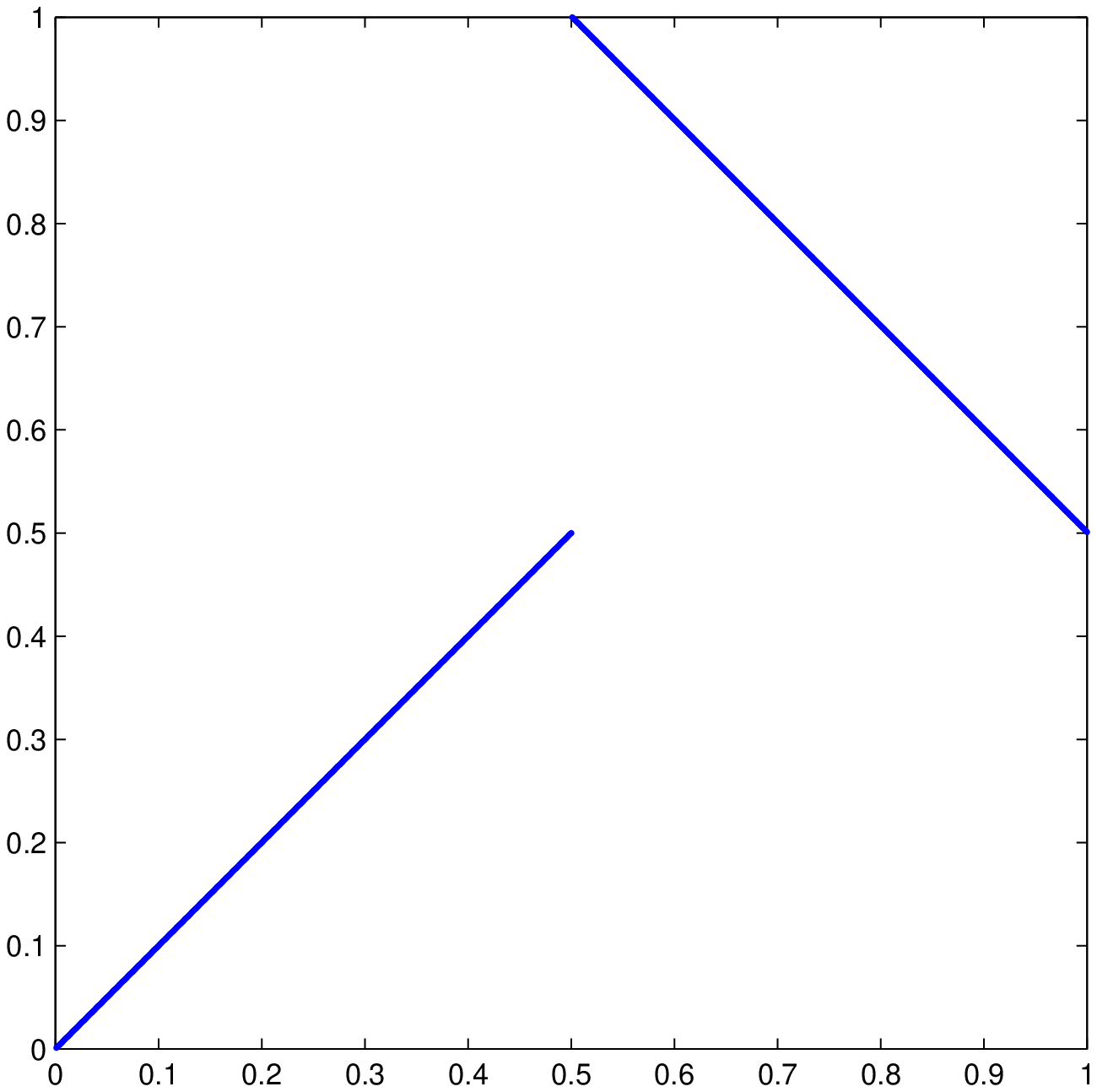} \includegraphics[width=0.26\textwidth]{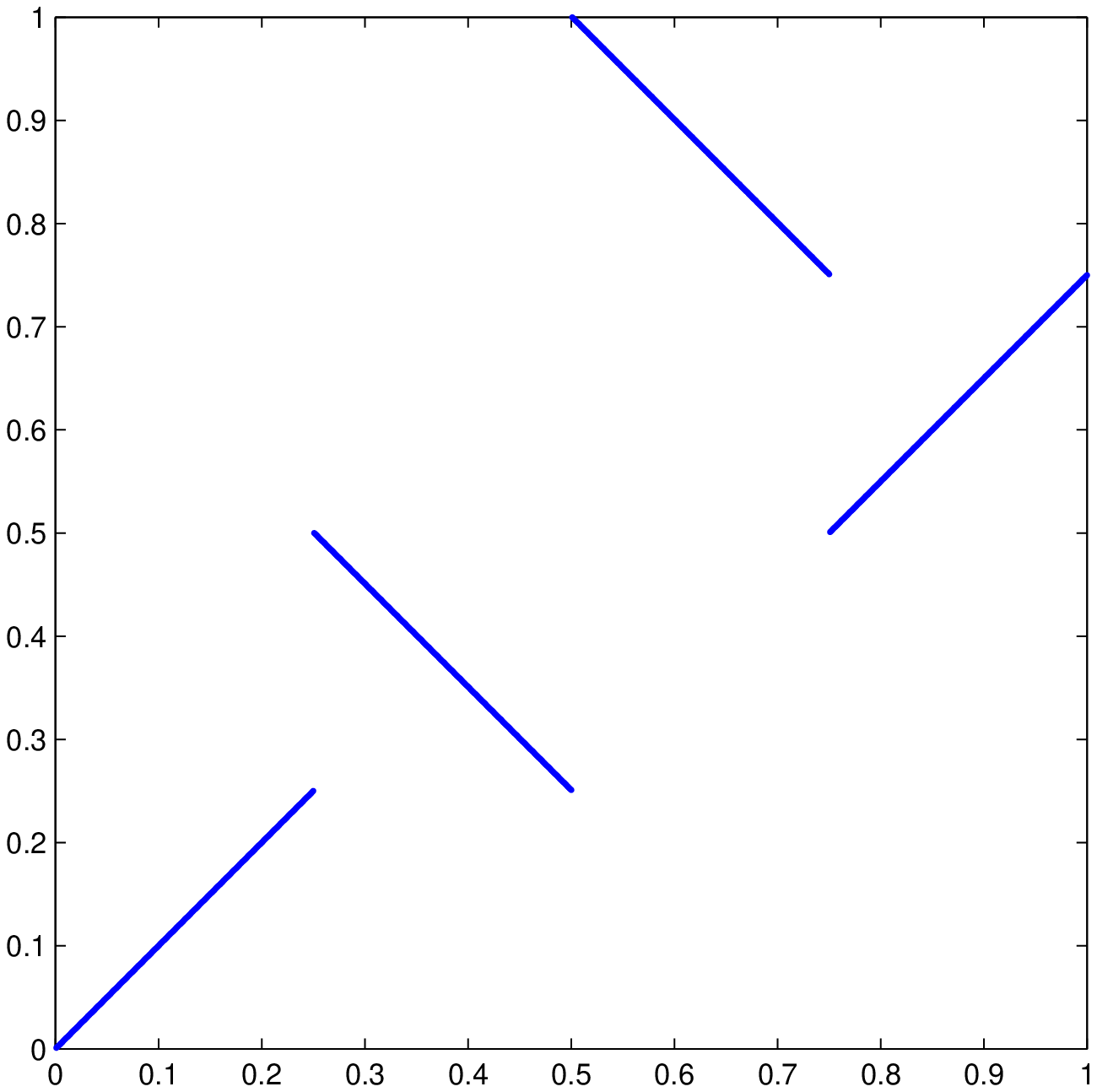} & \includegraphics[width=0.8\textwidth]{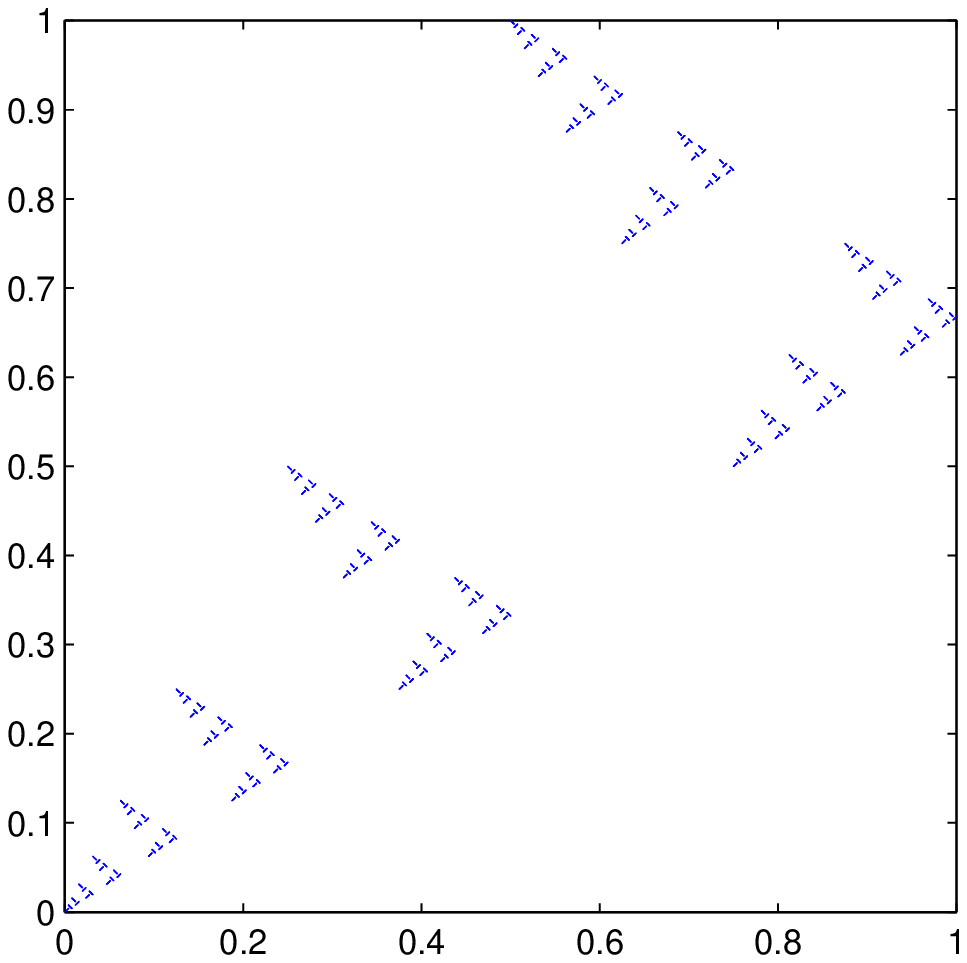}
  \end{tabular}
  \caption{$S_1$, $S_2$, and $S_8$}
  \end{figure}

  From construction, $F_n$ consists of $2^n$ stripes, each of width $\dfrac{1}{2^{n+1}}$.  On each of these stripes, the supports of $S_n$ and $S_{n-1}$ differ by a flip which implies that $\frac{\partial S_n}{\partial x}$ and $\frac{\partial S_{n-1}}{\partial x}$ are equal on the stripe except on two triangles of total area $\frac{1}{2}\paren{\frac{1}{2^{n+1}}}^2 = \frac{1}{2^{2n+3}}$ where $\abs{\frac{\partial S_n}{\partial x} - \frac{\partial S_{n-1}}{\partial x}} = 1$. Similarly, on each stripe of $F_n$, $\abs{\frac{\partial S_n}{\partial y} - \frac{\partial S_{n-1}}{\partial y}} = 1$ on two triangles of total area $\frac{1}{2^{2n+3}}$ and zero elsewhere. Therefore,
\[     \norm{S_n - S_{n-1}}^2 = \iint_{I^2}\paren{\frac{\partial S_n}{\partial x} - \frac{\partial S_{n-1}}{\partial x}}^2 + \paren{\frac{\partial S_n}{\partial y} - \frac{\partial S_{n-1}}{\partial y}}^2\,dx\,dy
      = \frac{1}{2^{n+2}}. \]
  Now, given $m<n$, $$\norm{S_n-S_m} \leq \sum_{k=m+1}^n\norm{S_k-S_{k-1}} = \frac{1}{2} \sum_{k=m+1}^n \frac{1}{\sqrt{2}^{k}} = \frac{\sqrt{2}^{-m-1}-\sqrt{2}^{-n-1}}{2-\sqrt{2}}$$ which converges to $0$ as $m,n\to\infty$. Since the set of copulas is complete with respect to the Sobolev norm (see p.~424 in \cite{darsow1995nc}), the Cauchy sequence $\set{S_n}$ converges to a copula $S$. It follows that $\norm{S}=1$. It can also be shown that the support of $S$ contains the graph of the pointwise limit $f$ of $f_n$.

  Finally, we shall show that the mutual complete dependence copula $S$ has support on the graph of a function discontinuous on the set of dyadic points in $[0,1]$.  In fact, it is straightforward to calculate the jump of $f$ at a dyadic point $\frac{k}{2^n}$ where $k$ is indivisible by $2$: $$\abs{f\paren{\frac{k}{2^n}+} - f\paren{\frac{k}{2^n}-}} = \frac{1}{2^{n}} - \frac{1}{2^{n+1}} + \frac{1}{2^{n+2}} - \dots = \frac{1}{3\cdot 2^{n-1}} > 0.$$ We note here that the support of $S$ is self-similar with Hausdorff dimension one.
\end{exam}


A surprising fact by Mikusinski, Sherwood and Taylor \cite[Theorem 3.1]{mikusinski1992sm} is that every copula, in particular the independence copula, can be approximated arbitrarily close in the uniform norm by a shuffle of Min. 
Consequently, the uniform norm cannot distinguish dependence structures among copulas.  However, if $\set{S_n}$ is a sequence of shuffles of Min converging in the Sobolev norm to a copula $C$, then it is necessary that $\norm{C}=1$, hence $C$ is a copula of two mutually completely dependent random variables. 
Conversely, one might ask whether any copula $C$ with $\norm{C}=1$ can be approximated arbitrarily close in the Sobolev norm by a shuffle of Min. 
We quote here without proof a result from \cite{chou1990frechet} which will be useful in answering the question. 
\begin{thm}[Chou and Nguyen \cite{chou1990frechet}]\label{thm:ChouNguyen}
  For every measure-preserving function $f$ over $[0,1]$, there exists a sequence of bijective piecewise linear measure-preserving functions  $\set{f_n}$ whose slopes are either $+1$ or $-1$ and such that $f_n$ converges to $f$ a.e.
\end{thm}

\begin{lem}\label{lem:norm-supp} \draftnote{Similar result for C.D. copulas ?}
  Let $C_1$ and $C_2$ be copulas with norm one which are supported on the graphs of $f_1$ and $f_2$, respectively. Then
  \begin{equation}
    \norm{C_1-C_2}^2 \leq {2}\norm{f_1-f_2}_{L^1}.
  \end{equation}
\end{lem}
\begin{proof}
 By assumption, for a.e.~$(x,y)$, $|\partial_1C_1(x,y)-\partial_1C_2(x,y)| = 1$ if and only if $y$ is between $f_1(x)$ and $f_2(x)$. Likewise, $|\partial_2C_1(x,y)-\partial_2C_2(x,y)| = 1$ if and only if $x$ is between $f_1^{(-1)}(y)$ and $f_2^{(-1)}(y)$ for a.e.~$(x,y)$. So
  \begin{align*}
    \norm{C_1-C_2}^2
    &= \int_0^1 \abs{f_1(x)-f_2(x)}^2\,dx + \int_0^1 \abs{f_1^{(-1)}(y)-f_2^{(-1)}(y)}^2\,dy\\
    &\leq \int_0^1 \abs{f_1(x)-f_2(x)}^2\,dx + \int_0^1 \abs{f_1^{(-1)}(y)-f_2^{(-1)}(y)}\,dy\\
    &= \norm{f_1-f_2}^2_{L^2} +  \norm{f_1-f_2}_{L^1} \leq 2\norm{f_1-f_2}_{L^1}.
  \end{align*}
\end{proof}

\begin{thm}\label{thm:C-S}
  For any copula $C$ with $\norm{C} = 1$, there exists a sequence of shuffles of Min $\set{S_n}$ such that $\norm{C-S_n} \to 0$.
\end{thm}
\begin{proof}
  Suppose $C$ is a copula with norm one and $C$ is supported on the graph of $f$. It follows from Theorem \ref{thm:suppcopulanormone} that $f$ is a measure-preserving bijection from $[0,1]$ onto itself. By Theorem \ref{thm:ChouNguyen}, one can construct a sequence of measure-preserving functions $\set{f_n}$ for which each $f_n$ is bijective piecewise linear with slopes $+1$ or $-1$ and $f_n$ converges to $f$ a.e.
  A corresponding sequence of shuffles of Min $\{ S_n \}$ can then be chosen so that the graph of $f_n$ is the support of $S_n$. By Lemma \ref{lem:norm-supp}, $\norm{C-S_n}^2 \leq 2\norm{f-f_n}_1$. Since $f-f_n\to 0$ a.e., an application of dominated convergence theorem shows that $\norm{f-f_n}_1\to 0$. Consequently, $S_n\to C$ in the Sobolev norm.
\end{proof}
\begin{remk}
  From the proof, it is worth noting that one can approximate a copula $C$ by only straight shuffles of Min whose slopes on all subintervals are $+1$. 
\end{remk}

\begin{cor} Let $U,V \in \mathfrak{C}$.
\begin{enumerate}
  \item If $\|U\|=1$ and $\|V\|=1$ then $\|U \ast V\|=1$. 
  \item if $\|U*V\|=1$ then $\|U\|=1$ if and only if $\|V\|=1$.
\end{enumerate}
\label{lem1}
\end{cor}

\begin{proof} 
1. Let $U,V \in \mathfrak{C}$ be such that $\|U\|=1$ and $\|V\|=1$. 
By Theorem \ref{thm:C-S}, there exist sequences $\set{S_n}$, $\set{T_n}$ of shuffles of Min such that $S_n \rightarrow U$ and $T_n \rightarrow V$ in the Sobolev norm.  Hence, with respect to the Sobolev norm, $S_n * T_n \rightarrow U*V$ by the joint continuity of the $*$-product.  Since a product of shuffles of Min is still a shuffle of Min, $\|U \ast V\|= 1.$

2. Let $U$ and $U \ast V$ be copulas of Sobolev norm 1. Since $\|U^T\|=\|U\|=1$, an application of 1.~yields $ \|V\|=\|U^T*(U*V)\|=1.$
\end{proof}

\section{Shuffles of Copulas and a Probabilistic Interpretation}
At least as soon as shuffles of Min were introduced in \cite{mikusinski1992sm}, the idea of simple shuffles of copulas was already apparent.  See, e.g., \cite[p.111]{mikusinski1991probabilistic}.  In \cite{durante2008sc}, Durante, Sarkoci and Sempi gave a general definition of shuffles of copulas
via a characterization of shuffles of Min in terms of a \emph{shuffling} $S_T\colon[0,1]^2\to [0,1]^2$ defined by $S_T(u,v) = \bigl(T(u),v\bigr)$ where $T\colon [0,1]\to [0,1]$. 
Before stating their results, let us recall the definition of push-forward measures. Let $f$ be a measurable function from a measure space $(\Omega,\EuFrak{F},\mu)$ to a measurable space $(\Omega_1,\EuFrak{F}_1)$. A \emph{push-forward of $\mu$ under $f$} is the measure $f*\mu$ on $(\Omega_1,\EuFrak{F}_1)$ defined by $f*\mu(A) = \mu(f^{-1}(A))$ for $A\in\EuFrak{F}_1$.
\begin{thm}[{\cite[Theorem 4]{durante2008sc}}]
  A copula $C$ is a shuffle of Min if and only if there exists a piecewise-continuous measure-preserving bijection $T\colon[0,1]\to[0,1]$ such that $\mu_C = S_T*\mu_M$.
\end{thm}
Dropping piecewise continuity of $T$, a \emph{generalized shuffle of Min} is defined as a copula $C$ whose induced measure is $\mu_C = S_T*\mu_M$ for some measure-preserving bijection $T\colon[0,1]\to[0,1]$. Replacing $M$ by a given copula $D$, a \emph{shuffle of $D$} is a copula $C$ whose induced measure is 
\begin{equation} \label{eq:pieceShuffle} 
  \mu_C = S_T*\mu_D
\end{equation}
for some piecewise-continuous measure-preserving bijection $T$. $C$ is also called the \emph{$T$-shuffle of $D$}. If the bijection $T$ is only required to be measure-preserving in \eqref{eq:pieceShuffle}, then $C$ is called a \emph{generalized shuffle of $D$}.

The following lemma will be useful in our investigation.
\begin{lem}\label{lem:1}
  Let $T$ be a measure-preserving bijection on $[0,1]$ and $C$ be a copula defined by \[C(x,y) = S_{T}\ast\mu_M\paren{[0,x]\times[0,y]}\quad\text{for } x,y\in[0,1].\] Then 
  the copula $C$, or equivalently its induced measure $\mu_C = S_{T}\ast\mu_M$, is supported on the graph of $T^{-1}$.  Moreover, the converse also holds, i.e.~if $C$ is supported on the graph of a measure-preserving bijection $T$ then $\mu_C = S_{T^{-1}}\ast\mu_M$.
\end{lem}
\begin{proof}
  Let $[a,b]\times[c,d]$ be a closed rectangle in $\R^2$ and $S_T$ be the map on $[0,1]^2$ associated with a given measure-preserving bijection $T$ on $[0,1]$, i.e.~$S_T(u,v)=(T(u),v)$. So $S_T^{-1}\paren{[a,b]\times[c,d]} = \paren{T^{-1}[a,b]}\times[c,d]$ and, by definition of the push-forward measure,
  \begin{align*}
    S_T\ast\mu_M\paren{[a,b]\times[c,d]}
     &= \mu_M\paren{S_T^{-1}\paren{[a,b]\times[c,d]}}\\
     &= \mu_M\paren{\paren{T^{-1}[a,b]}\times[c,d]}
     = m\paren{\paren{T^{-1}[a,b]}\cap[c,d]}.
  \end{align*}
  Thus, $S_T\ast\mu_M\paren{[a,b]\times[c,d]} = 0$ if and only if  the projection of $\graph(T^{-1})\cap\paren{[a,b]\times[c,d]}$ onto $[c,d]$ has measure zero. Consequently, since Borel measurable subsets of $[0,1]^2$ are generated by rectangles, the desired result is obtained.
\end{proof}

\begin{thm}
  A copula $C$ is a generalized shuffle of Min if and only if $\norm{C}=1$.
\end{thm}
\begin{proof}
  ($\Rightarrow$) Let $C$ be a generalized shuffle of Min, i.e.~there exists a measure preserving bijection $T$ on $[0,1]$ such that $\mu_C = S_T\ast\mu_M$. By Theorem \ref{thm:ChouNguyen}, there is a sequence $\set{T_n}$ of piecewise-continuous measure-preserving bijection on $[0,1]$ such that $T_n\to T$ a.e. So $C_n(x,y) = S_{T_n}\ast\mu_M\paren{[0,x]\times[0,y]}$ defines a sequence of shuffles of Min. We claim that $\norm{C_n-C}\to 0$. In fact, by Lemma \ref{lem:1}, $C=S_T*\mu_M$ and $C_n = S_{T_n}*\mu_M$ 
  are supported on the graphs of $T^{-1}$ and $T_{n}^{-1}$ respectively.  Now, Lemma \ref{lem:norm-supp} implies that $\norm{C_n-C}^2\leq 2\norm{T^{-1}-T_n^{-1}}_{L^1}$ which converges to $0$ as a result of the Lusin-Souslin Theorem (see,  e.g., \cite[Corollary 15.2]{Ke}) which states that a Borel measurable injective image of a Borel set is a Borel set 
  and the dominated convergence theorem. Therefore, $C_n\to C$ in the Sobolev norm.

  ($\Leftarrow$) Let $C$ be a copula with $\norm{C}=1$. Then Theorem \ref{thm:suppcopulanormone} gives a measure-preserving bijection $f$ whose graph is the support of $C$.  So Lemma \ref{lem:1} implies that $\mu_C = S_{f^{-1}}\ast\mu_M$.
\end{proof}

\begin{thm} \label{thm:mu*nu}
  If $\mu$ and $\nu$ are doubly stochastic measures on $[0,1]^2$ then \[\mu*\nu (I\times J) = \int_0^1\partial_2\mu(I,t)\partial_1\nu(t,J)\,dt\] induces a doubly stochastic (Borel) measure $\mu*\nu$ on $[0,1]^2$, where \[\partial_2\mu(I,t) = \frac{d}{dt}\mu(I\times[0,t])\quad\text{and}\quad \partial_1\nu(t,J) = \frac{d}{dt}\nu([0,t]\times J).\] Furthermore, if $A$ and $B$ are copulas and $\mu_A$ and $\mu_B$ denote their doubly stochastic measures then \begin{equation}\label{eqn:1} \mu_{A*B} = \mu_A*\mu_B. \end{equation}
\end{thm}
\begin{proof}
  We shall prove only \eqref{eqn:1} which shows that $\mu*\nu$ is a doubly stochastic measure when the measures $\mu$ and $\nu$ are doubly stochastic and inducible by copulas. 
  Let $A$ and $B$ be copulas and $I=[a_1,a_2], J=[b_1,b_2] \subseteq [0,1]$. Then
  \begin{align*}
    \mu_{A*B}(I\times J) 
      &= \int_0^1 \left[\partial_{2}A(a_2,t)\partial_1B(t,b_2) - \partial_{2}A(a_1,t)\partial_1B(t,b_2)\right.\\
       &\quad - \partial_{2}A(a_2,t)\partial_1B(t,b_1) + \partial_{2}A(a_1,t)\partial_1B(t,b_1)\Large] \,dt\\
      &= \int_0^1 \partial_2\paren{A(a_2,t)-A(a_1,t)}\partial_1\paren{B(t,b_2)-B(t,b_1}\,dt\\
      &= \int_0^1 \frac{d}{dt}\mu_A(I\times[0,t])\frac{d}{dt}\mu_B([0,t]\times J)\,dt\\
      &= \mu_A*\mu_B(I\times J).
  \end{align*}
  The usual measure-theoretic techniques allow to extend this result to the product of all Borel sets.
\end{proof}

\begin{lem}\label{lem:S_Tassoc}
  Let $T$ be a measure-preserving bijection on $[0,1]$ and $\mu$, $\nu$ be doubly stochastic measures on $[0,1]^2$. Then
  \[S_T*(\mu * \nu) = (S_T*\mu) * \nu.\]
\end{lem}
\begin{proof}
  Let $I$ and $J$ be Borel sets in $[0,1]$. Then
  \begin{align*}
    \paren{S_T*(\mu * \nu)}(I\times J) &= (\mu * \nu)\paren{S_T^{-1}(I\times J)} = (\mu * \nu)\paren{T^{-1}(I)\times J}\\
      &= \int_0^1 \partial_2\mu\paren{T^{-1}(I),t}\partial_1\nu(t,J)\,dt\\
      &= \int_0^1 \partial_2(S_T*\mu)\paren{I,t}\partial_1\nu(t,J)\,dt\\
      &= \paren{(S_T*\mu) * \nu}(I\times J).
  \end{align*}
\end{proof}

\begin{thm}\label{thm:socChar}
  Let $C$ and $D$ be bivariate copulas. Then
  \begin{enumerate}
    \item $C$ is a shuffle of $D$ if and only if there exists a shuffle of Min $A$ such that $C=A\ast D$;
    \item $C$ is a generalized shuffle of $D$ if and only if there exists a generalized shuffle of Min $A$ such that $C=A\ast D$. 
  \end{enumerate}
\end{thm}
\begin{proof}
  We shall only prove 2.~since 1.~is just a special case.

  ($\Rightarrow$) If $C$ is a shuffle of $D$, i.e.~$\mu_C = S_T*\mu_D$ for some measure-preserving bijection $T$ of $[0,1]$, then the copula $A$ defined by $\mu_A = S_T*\mu_M$ is a shuffle of Min by Theorem \ref{thm:mu*nu}. Then \[\mu_C = S_T*\mu_D = S_T*\mu_{M*D} = S_T*\paren{\mu_M*\mu_D} = (S_T*\mu_M)*\mu_D = \mu_A*\mu_D\] which means that $C = A*D$.

  ($\Leftarrow$) If $C = A*D$ for some copula $A$ with $\norm{A}=1$ then $\mu_A = S_T*\mu_M$ for some measure-preserving bijection $T$ and
  \[\mu_C = \mu_{A*D} = \mu_A*\mu_D = (S_T*\mu_M)*\mu_D = S_T*\paren{\mu_M*\mu_D} = S_T*\mu_{M*D} = S_T*\mu_D.\]
  Note the repeated uses of Theorem \ref{thm:mu*nu} and Lemma \ref{lem:S_Tassoc} in both derivations.
\end{proof}
\begin{remk}
Since $\Pi$ is the only null element of $*$ (see \cite{darsow1992copulas}), it follows easily from Theorem \ref{thm:socChar} that $\Pi$ is the only copula which is invariant under shuffling by generalized shuffles of Min. This is a result first proved in \cite[Theorem 10]{durante2008sc}.
\end{remk}
Even though all generalized shuffles of Min have equal unit norm, not all shuffles of $C$ have the same norm. Here is a class of examples.
\begin{exam} \label{exam:shufflediffnorm}
  For $0\leq \alpha <1 $, let $S_\alpha$ denote the straight shuffle of Min whose support is on the main diagonals of the squares $[0,\alpha]\times[1-\alpha,1]$ and $[\alpha,1]\times[0,1-\alpha]$. Then by straightforward computations, for any copula $C$, \[S_\alpha*C(x,y) = \begin{cases}
    C(x+1-\alpha,y) - C(1-\alpha,y) & \text{if } 0\leq x\leq \alpha,\\
    y-C(1-\alpha,y) + C(x-\alpha,y) & \text{if } \alpha< x\leq 1,
  \end{cases}\]
  and
  \begin{align}
    \norm{S_\alpha*C}^2 &= \norm{C}^2 + \int_0^1\bigl(\partial_2C(1-\alpha,y)-(1-\alpha)\bigr)^2\,dy \notag \\
      &\quad -2\int_0^1\int_0^1\partial_2C(x,y)\bigl(\partial_2C(1-\alpha,y)-(1-\alpha)\bigr)\,dx\,dy.\label{eqn:2}
  \end{align}
  Let us now consider the Farlie-Gumbel-Morgenstern (FGM) copulas $C_\theta$, $\theta\in[-1,1]$, defined by $C_\theta(x,y) = xy + \theta xy(1-x)(1-y)$. Then
  \[\int_0^1\bigl(\partial_2C_\theta(1-\alpha,y)-(1-\alpha)\bigr)^2\,dy = \frac{\theta^2\alpha^2(1-\alpha)^2}{3}\]
  and
  \[2\int_0^1\int_0^1\partial_2C_\theta(x,y)\bigl(\partial_2C_\theta(1-\alpha,y)-(1-\alpha)\bigr)\,dx\,dy = \frac{2\theta^2\alpha(1-\alpha)}{9}.\]
  So that $\norm{S_\alpha*C_\theta}^2 = \norm{C_\theta}^2 - \frac{\theta^2\alpha(1-\alpha)}{3}\paren{\frac{2}{3}-\alpha(1-\alpha)}$ which is equal to $\norm{C_\theta}^2$ only if $\theta=0$ or $\alpha=0$ or $1$. For each $\theta\neq 0$, $\norm{C_\theta}^2-\norm{S_\alpha*C_\theta}^2$ is maximized when $\alpha=\frac{1}{2}$ and the maximum value is $\frac{5\theta^2}{12^2}$.
\end{exam}



\begin{prop}[{\cite{darsow1992copulas}}, p.~610] If $Z$ and $Y$ are conditionally independent given $X$, then $C_{Z,Y} = C_{Z,X}\ast C_{X,Y}.$ \label{Da-con-ind}
\end{prop}

\begin{prop} 
Let $h \colon \mathbb{R} \rightarrow \mathbb{R}$ be Borel measurable and $X,Y$ be random variables. Then $h(X)$ and $Y$ are conditionally independent given $X$.\label{special-case}
\end{prop}
\begin{proof} 
Since $h$ is Borel measurable, $h(X)$ is measurable with respect to $\sigma(X)$, the $\sigma$-algebra generated by $X$. Hence, by properties of conditional expectations,
\begin{align*}
E(I_{h(X) \le a}|X)(\omega)\cdot E(I_{Y \le b}|X)(\omega)&=I_{h(X) \le a}(\omega)\cdot E(I_{Y \le b}|X)(\omega)\\
&=E(I_{h(X) \le a}\cdot I_{Y \le b}|X)(\omega)
\end{align*}
for all $\omega \in \Omega$. This completes the proof.
\end{proof}

\begin{cor} \label{cor:star-shuffling}
Let $f,g \colon \mathbb{R} \rightarrow \mathbb{R}$ be Borel measurable functions. Then \begin{center} $C_{f(X),X}\ast C_{X,Y}\ast C_{Y,g(Y)}= C_{f(X),g(Y)}$\end{center} for all random variables $X,Y$.\label{decompose}
\end{cor}

\begin{proof} Since $f$ and $g$ are Borel measurable, by Propositions \ref{Da-con-ind} and \ref{special-case}, we have 
\begin{equation} \label{eq-2}
C_{f(X),Y} = C_{f(X),X}\ast C_{X,Y} \quad \text{and}\quad
C_{g(Y),X} = C_{g(Y),Y}\ast C_{Y,X}
\end{equation}
for all random variables $X,Y$. Transposing both sides of \eqref{eq-2}, we obtain $C_{X,g(Y)} = C_{X,Y} * C_{Y,g(Y)}.$
Then, we have 
\begin{equation}
  C_{f(X),g(Y)} = C_{f(X),X}*C_{X,g(Y)} 
= C_{f(X),X}*C_{X,Y}*C_{Y,g(Y)}.
\end{equation}
\end{proof}

\begin{defn}
Let $U,V \in \Inv\mathfrak{C}$, the set of invertible copulas or, equivalently, the set of copulas with unit Sobolev norm. A \emph{shuffling map} $S_{U,V}$ is a map on $\spanof\mathfrak{C}$ defined by \begin{center} $S_{U,V}(A)=U \ast A \ast V$. \end{center}
\end{defn}

The motivation behind the word \lq\lq shuffling\rq\rq~comes from the fact that a shuffling image of a copula is a two-sided generalized shuffle of the copula. 
Note that $C \text{ is invertible} \Leftrightarrow \norm{C}=1 \Leftrightarrow C \text{ is a generalized shuffle of Min.}$

\begin{lem} Let $X,Y$ be continuous random variables and $U,V \in \Inv\mathfrak{C}$. Then the following statements hold:
\begin{enumerate}
\item $X$ and $Y$ are independent if and only if $S_{U,V}(C_{X,Y})= \Pi$.
\item $X$ is completely dependent on $Y$ or vice versa if and only if $S_{U,V}(C_{X,Y})$ is a complete dependence copula.
\item $X$ and $Y$ are mutually completely dependent if and only if $S_{U,V}(C_{X,Y})$ is a mutual complete dependence copula.
\end{enumerate}\label{dependence-invariant}
\end{lem}

\begin{proof} 
1. This clearly follows from the fact that $\Pi$ is the zero element in $(\mathfrak{C},*)$.

2. With out loss of generality, let us assume that $Y$ is completely dependent on $X$, i.e.~there exists a Borel measurable transformation $h$ such that $Y=h(X)$ with probability one. 
Let $f$ and $g$ be Borel measurable bijective transformations on $\R$ such that $U=C_{f(X),X}$ and $V=C_{Y,g(Y)}$.  By Corollary \ref{decompose}, we have \begin{center} $S_{U,V}(C_{X,Y})= C_{f(X),X} \ast C_{X,Y} \ast C_{Y,g(Y)}= C_{f(X),g(Y)}$. \end{center} Thus, it suffices to show that $g(Y)$ is completely dependent on $f(X)$. From $Y=h(X)$ with probability one, $g(Y)=(g\circ h)(X) = (g \circ h \circ f^{-1})(f(X))$ with probability one. It is left to show that $f^{-1}$ is Borel measurable. This is true because of Lusin-Souslin Theorem (see,  e.g., \cite{Ke}, Corollary 15.2) which states that a Borel measurable injective image of a Borel set is a Borel set.  The converse automatically follows because the inverse of a shuffling map is still a shuffling map.

3. The proof is completely similar to above except that the function $h$ is also required to be bijective.
\end{proof}

Corollary \ref{decompose} implies that a shuffling image of a copula $C_{X,Y}$ is a copula of transformed random variables $C_{f(X),g(Y)}$ for some Borel measurable bijective transformations $f$ and $g$. Together with the above lemma, we obtain the following theorem.

\begin{thm} Let $X$ and $Y$ be continuous random variables. Let $f$ and $g$ be any Borel measurable bijective transformations of the random variables $X$ and $Y$, respectively. Then $X$ and $Y$ are independent, completely dependent or mutually completely dependent if and only if $f(X)$ and $g(Y)$ are independent, completely dependent or mutually completely dependent, respectively.
\end{thm}

The above theorem suggests that shuffling maps preserve stochastic properties of copulas. In the next section, we contruct a norm which, in some sense, also preserves stochastic properties of copulas.

\section{The $*$-norm}

Our main purpose is to construct a norm under which shuffling maps are isometries and then derive its properties.

\begin{defn} 
Define a map $\|\cdot\|_*:\spanof\mathfrak{C} \rightarrow [0,\infty)$, by \begin{displaymath} \|A\|_* = \sup_{U,V \in \Inv{\mathfrak{C}}} \|U\ast A\ast V\|. \end{displaymath}
\end{defn}
By straightforward verifications, $\norm{\cdot}_*$ is a norm on $\spanof\mathfrak{C}$, called the \emph{$*$-norm}.
Moreover, it is clear from the definition that $\|A\| \le \|A\|_*$ for all $A \in \spanof{\mathfrak{C}}$.

The following proposition summarizes basic properties of the $*$-norm.  Observe that properties 2.--4.~are the same as those for the Sobolev norm.

\begin{prop} \label{prop:*norm}
 Let 
 $C \in \mathfrak{C}$. Then the following statements hold.
\begin{enumerate}
  \item $\|C\|_*=1$ if $\|C\|=1$.
  \item $\|C\|_*^2=\frac{2}{3}$ if and only if $C = \Pi$.
  \item \label{prop:*norm3} $\|C-\Pi\|_*^2=\|C\|_*^2-\frac{2}{3}$.
  \item $\|A^T\|_*=\|A\|_*$ for all $A \in \spanof{\mathfrak{C}}$.
\end{enumerate}
\label{props}
\end{prop}

\begin{proof}
1.~is a consequence of the inequality $\|C\| \le \|C\|_*\le 1$. 
2.~
follows from the fact that $\Pi$ is the zero of $(\mathfrak{C},*)$.  
To prove 3., we first observe that 
\begin{center}$\|U*(C-\Pi)*V\|^2=\|U*C*V-\Pi\|^2=\|U*C*V\|^2-\frac{2}{3}$\end{center} for all $U,V \in \Inv\mathfrak{C}$. The result follows by taking supremum over $U,V \in \Inv\mathfrak{C}$ on both sides.  Finally, 
using the facts that $\|U^T\|=\|U\|$ for all $U \in \mathfrak{C}$,
\begin{align*} 
  \|A^T\|_* &= \sup_{U,V \in \Inv{\mathfrak{C}}} \|U*A^T*V\|
= \sup_{U,V \in \Inv{\mathfrak{C}}} \|V^T*A*U^T\|\\
&= \sup_{U^T,V^T \in \Inv{\mathfrak{C}}} \|V^T*A*U^T\|
= \sup_{U,V \in \Inv{\mathfrak{C}}} \|U*A*V\|
= \|A\|_*.
\end{align*}
\end{proof}

\begin{thm} 
Let $A\in\spanof\mathfrak{C}$ and $U\in\Inv\mathfrak{C}$. Then $\norm{U*A}_* = \norm{A}_* = \norm{A*U}_*$.  Therefore, shuffling maps are isometries with respect to the $*$-norm.
\end{thm}
\begin{proof} 
We shall prove only one side of the equation as the other can be proved in a similar fashion. 
Let $A \in \spanof\mathfrak{C}$ and $U_o \in \Inv\mathfrak{C}$. Then by Corollary \ref{lem1}, for any $C \in \mathfrak{C}$, 
$U_o*C \in \Inv{\mathfrak{C}}$ if and only if $C \in \Inv{\mathfrak{C}}$.
Hence,
$\|U_o*A\|_* = \sup_{U,V \in \Inv{\mathfrak{C}}} \|(U*U_o)*A*V)\|
= \sup_{U,V \in \Inv{\mathfrak{C}}} \|U*A*V\|
= \|A\|_*.$
\end{proof}


\begin{exam} 
  From Example \ref{exam:shufflediffnorm}, 
  let $\alpha \in (0,1)$, $\theta\in[-1,1]\setminus\set{0}$ and $A_{\theta} = S_{1/2}\ast C_{\theta}$. Then $S_{1/2}\ast A_{\theta} = C_{\theta}$. 
  Since $\|A_{\theta}\| < \|C_{\theta}\|$ for any $\theta \neq 0$. Then
\begin{center} $\|A_{\theta}\|_* \ge \|S_{1/2} \ast A_{\theta}\|=\|C_{\theta} \|>\|A_{\theta}\|.$ \end{center}
Hence, the Sobolev norm and the $\ast$-norm are distinct.
\label{norms-distinct}
\end{exam}

\begin{exam} \label{exam:EqualNorms}
  Let $\alpha\in[0,1]$ and $C$ be a copula.  Recall that one can show using only the property $\norm{C-\Pi}^2 = \norm{C}^2-\frac{2}{3}$ of the norm $\norm{\cdot}$ (see \cite{siburg2008spc}) that $\norm{\alpha C + (1-\alpha)\Pi}^2 = \alpha^2\paren{\norm{C}^2-\frac{2}{3}} + \frac{2}{3}$.  Since the $*$-norm shares this same property with the Sobolev norm (see Proposition \ref{prop:*norm}(\ref{prop:*norm3})), we also have \[\norm{\alpha C + (1-\alpha)\Pi}_*^2 = \alpha^2\paren{\norm{C}_*^2-\frac{2}{3}} + \frac{2}{3}.\] So $\|\alpha C +(1-\alpha)\Pi\|^2_* = \|\alpha C +(1-\alpha)\Pi\|^2$ for all copulas $C$ satisfying $\norm{C}_* = \norm{C}$.  In particular, the Sobolev norm and the $\ast$-norm coincide on the family of convex sums of an invertible copula and the product copula, where the norms are equal to $(\alpha^2+2)/3$.
\end{exam}

\begin{lem} \label{lem:shufflingCD}
  Let $A\subseteq [0,1]$ be a Borel measurable set. Define the function $s_A\colon [0,1] \to [0,1]$ by
  \begin{equation} \label{eq:shufflingCD}
    s_A(x) = \begin{cases} m([0,x]\cap A) & \text{if } x\in A,\\
                           m(A) + m([0,x]\setminus A) & \text{if } x\notin A. \end{cases}
  \end{equation}                      
Then $s_A$ is measure-preserving and \emph{essentially invertible} in the sense that there exists a Borel measurable function $t_A$ for which $s_A\circ t_A(x) = x = t_A\circ s_A(x)$ a.e.~$x\in[0,1]$.  Such a $t_A$ is called an \emph{essential inverse} of $s_A$.
\end{lem}

\begin{proof} 
Clearly, $s_A$ is Borel measurable.

\noindent
\textbf{$\bullet$ $s_A$ is measure-preserving:} It suffices to prove that $m(s_A^{-1}[0,b])=m([0,b])$ for all $b \in [0,1]$. Now if $b \leq m(A)$, then 
\begin{equation*}
s_A^{-1}[0,b] = \{ x\in A\colon s(x) \in [0,b] \} 
=\set{ x \in A\colon m(A \cap [0,x]) \leq b }.
\end{equation*}
By continuity of $m$, there exists a largest $x_0$ such that $m(A \cap [0,x_0]) = b$. Then 
$s_A^{-1}[0,b] 
= A \cap \{x\colon m(A \cap [0,x]) \leq b\} = A \cap [0,x_0]$. Therefore, $m(s_A^{-1}[0,b]) = m(A \cap [0,x_0]) = b$.
The case where $b > m(A)$ can be proved similarly.

\noindent
\textbf{$\bullet$ $s_A$ is essentially invertible:}
  Using continuity of $m$, we shall define an auxiliary function $t_A$ on $[0,1]$ as follows.  If $y\leq m(A)$, there exists a corresponding $x\in A$ such that $m\paren{[0,x]\cap A} = y$.  If $y > m(A)$, there exists a corresponding $x\notin A$ such that $m(A) + m\paren{[0,x]\setminus A} = y$.  In these two cases, we define $t_A(y) = x$.  Generally, $t_A$ is not unique as there might be many such $x$'s.  
%


We shall show that $s_A$ is injective outside a Borel set of measure zero by proving that $\iota \equiv t_A\circ s_A$ is the identity map on $[0,1]\setminus Z$ for some Borel set $Z$ of measure zero.  If $x\in A$ then $s_A(x)\leq m(A)$ so that 
$\iota(x) \in A$ and $m([0,\iota(x)]\cap A) = s_A(x) = m([0,x]\cap A)$.  Similarly, if $x\notin A$ then $s_A(x)\geq m(A)$, $\iota \notin A$ and  $m([0,\iota(x)]\setminus A) = m([0,x]\setminus A)$. 
Consider the set of all $x\in A$ for which $\iota(x)\neq x$.  If $\iota(x)<x$ then $m\paren{(\iota(x),x]\cap A} = 0$ which implies that 
\begin{equation} \label{eq:Density} 
  \lim_{r\to 0^+} \dfrac{m\paren{(x-r,x+r)\cap A}}{2r} = \lim_{r\to 0^+} \dfrac{m\paren{(x,x+r)\cap A}}{2r} \leq \dfrac{1}{2}
\end{equation} 
wherever exists.
Now, a simple application of the Radon-Nikodym theorem (see, e.g., \cite{folland99}) yields that the limit in \eqref{eq:Density} is equal to $1$ for all $x\in A\setminus Z$ where $Z$ is a Borel set of measure zero.  Therefore, $\set{x\in A\colon \iota(x)<x}$ has zero measure.  Similarly, $\set{x\in A\colon \iota(x)>x}$ is a null set and hence $m\paren{\set{x\in A\colon \iota(x)\neq x}} = 0$.  Therefore, $\iota(x)=t_A\circ s_A (x)=x$ except possibly on a Borel set $Z$ of measure zero.  

By the Lusin-Souslin Theorem (see,  e.g., \cite[Corollary 15.2]{Ke}), the injective Borel measurable function $s_A$, mapping a Borel set $[0,1]\setminus Z$ onto a Borel set $s_A([0,1]\setminus Z)$, has a Borel measurable inverse, still denoted by $t_A$.  Now, since $s_A$ is measure-preserving, its range which is the domain of $t_A$ has full measure. This guarantees that $t_A$ can be extended to a Borel measurable function on $[0,1]$ which is an essential inverse of $s_A$.
\end{proof}

\begin{thm} \label{thm:cdnorm1}
   Let $X,Y$ be random variables on a common probability space for which $Y$ is completely dependent on $X$ or $X$ is completely dependent on $Y$.  Then $\|C_{X,Y}\|_* = 1$.
\end{thm}
\begin{proof}
Assume that $Y$ is completely dependent on $X$.  Then $C=C_{X,Y}$ is a complete dependence copula for which $C = C_{U,f(U)}$ where $U$ is a uniform random variable on $[0,1]$ and $f \colon [0,1] \rightarrow [0,1]$ is a measure-preserving Borel function. Note that $f(U)$ is also a uniform random variable and that $C$ is left invertible.

As the first step, we shall construct an invertible copula $S_1$ such that $S_1*C$ is supported in the two diagonal squares $[0,1/2]^2 \cup [1/2,1]^2$. 
Let $A = f^{-1}([0,1/2])$, denote $s=s_A$ as defined in \eqref{eq:shufflingCD} and put $S_1=C_{s(U),U}$. By Lemma \ref{lem:shufflingCD}, $s$ is invertible a.e.~and hence $S_1$ is invertible. \draftnote{Use Lemma \ref{lem:shufflingCD} here} By Corollary \ref{cor:star-shuffling}, \[S_1*C=C_{s(U),U}*C_{U,f(U)}=C_{V,(f \circ s^{-1})(V)}\] where $V = s(U)$ is still a uniform random variable on $[0,1]$. It is left to verify that the support of $S_1*C$ lies entirely in the two diagonal squares which can be done by showing that the graph of $f \circ s^{-1}$ is contained in the area.  In fact, since $s^{-1}([0,\frac{1}{2}]) \subseteq A$, it follows that $(f \circ s^{-1})([0,\frac{1}{2}]) \subseteq f(A) \subseteq [0,\frac{1}{2}]$. The inclusion $(f \circ s^{-1})([\frac{1}{2},1]) \subseteq [\frac{1}{2},1]$ can be shown similarly.  As a consequence, $S_1*C$ can be written as an ordinal sum of two copulas, $C_1$ and $C_2$, with respect to the partition $\{[0,\frac{1}{2}],[\frac{1}{2},1]\}$.  Since left-multiplying a copula $C$ by an invertible copula amounts to shuffling the first coordinate of $C$, it follows that $C_1$ and $C_2$ are still supported on closures of graphs \draftnote{ess. closures?} of measure-preserving functions. 

Next, we apply the same process to $C_1$ and $C_2$ which yields invertible copulas $A_1$ and $A_2$ for which $A_1*C_1$ and $A_2*C_2$ are both supported in $[0,\frac{1}{2}]^2 \cup [\frac{1}{2},1]^2$ and define $S_2$ to be the ordinal sum of $A_1$ and $A_2$ with respect to the partition $\{[0,\frac{1}{2}],[\frac{1}{2},1]\}$. $S_2$ is again an invertible copula. Then the support of $S_2 * S_1 * C$ is contained in the four diagonal squares $\bigcup_{i=1}^4 \brac{\frac{i-1}{4},\frac{i}{4}}^2$.  Therefore, $S_2 * S_1 * C$ is an ordinal sum with respect to the partition $\{[0,\frac{1}{4}],[\frac{1}{4},\frac{1}{2}],[\frac{1}{2},\frac{3}{4}],[\frac{3}{4},1]\}$ of four copulas each of which is supported on the closure of graph \draftnote{ess. closures?} of a measure-preserving function.  By successively applying this process, we obtain a sequence of invertible copulas $\{B_n\}_{n=1}^{\infty}$, defined by $B_n = S_n*\dots*S_2*S_1$, such that the support of $B_n*C$ is a subset of the $2^n$ diagonal squares. So $B_n*C \rightarrow M$ pointwise outside the main diagonal and so are their partial derivatives. Hence $\|B_n*C\|\rightarrow 1$. Thus $\|C\|_* =1$.

If $X$ is completely dependent on $Y$ then $C=C_{X,Y}$ is right invertible and similar process where suitably chosen $S_n$'s are multiplied on the right yields a sequence $\set{B_n}$ of invertible copulas such that $\|C*B_n\|\to 1$ as desired
\end{proof}

\begin{cor} \label{cor:L*R}
  Let $L$ and $R$ be left invertible and right invertible copulas, respectively. Then $\|L*R\|_* = 1$.
\end{cor}
\begin{proof}
From Theorem \ref{thm:cdnorm1}, there exist sequences of invertible copulas $S_n$ and $T_n$ such that $S_n*L \rightarrow M$ and $R*T_n \rightarrow M$ in the Sobolev norm.  By the joint continuity of the $*$-product with respect to the Sobolev norm, $S_n*L*R*T_n \rightarrow M*M = M$. Therefore, $\|L*R\|_*=1$.
\end{proof}

Let us give some examples of copulas of the form $L*R$. Consider a copula $C$, $C^T$ and $C*C^T$ whose supports are shown in the figure below.
\begin{figure}[htb]
\psset{unit=2.8cm}
\begin{center}
\begin{pspicture*}(-0.25,-0.25)(4.2,1.2)
\psline[linecolor=black](0,1)(1,1)
\psline[linecolor=black](1,0)(1,1)
\psline[linecolor=black](0,0)(0,1)
\psline[linecolor=black](0,0)(1,0)

\psline[linecolor=black](1.5,1)(2.5,1)
\psline[linecolor=black](2.5,0)(2.5,1)
\psline[linecolor=black](1.5,0)(1.5,1)
\psline[linecolor=black](1.5,0)(2.5,0)

\psline[linecolor=black](3,1)(4,1)
\psline[linecolor=black](4,0)(4,1)
\psline[linecolor=black](3,0)(3,1)
\psline[linecolor=black](3,0)(4,0)

\psline[linecolor=black,linestyle=dotted](0.5,0)(0.5,1)

\psline[linecolor=black,linestyle=dotted](1.5,0.5)(2.5,0.5)

\psline[linecolor=black,linestyle=dotted](3,0.5)(4,0.5)
\psline[linecolor=black,linestyle=dotted](3.5,0)(3.5,1)

\psplot[linecolor=blue,plotpoints=400]{0}{0.5}{2 x mul}
\psplot[linecolor=blue,plotpoints=400]{0.5}{1}{2 2 x mul sub}

\psplot[linecolor=blue,plotpoints=400]{1.5}{2.5}{1.75 x 0.5 mul sub}
\psplot[linecolor=blue,plotpoints=400]{1.5}{2.5}{x 0.5 mul 0.75 sub}

\psplot[linecolor=blue,plotpoints=400]{3}{4}{x 3 sub}
\psplot[linecolor=blue,plotpoints=400]{3}{4}{1 x 3 sub sub}

\commentout{  
\uput{0.2}[270](0.5,0){$\frac{1}{2}$}
\uput{0.2}[180](1.5,0.5){$\frac{1}{2}$}
\uput{0.2}[270](3.5,0){$\frac{1}{2}$}
\uput{0.2}[180](3,0.5){$\frac{1}{2}$}
}
\end{pspicture*}
\end{center}
  \caption[ ]{the supports of $C$, $C^T$ and $C*C^T$, respectively} 
  \label{fig:example1}
\end{figure}
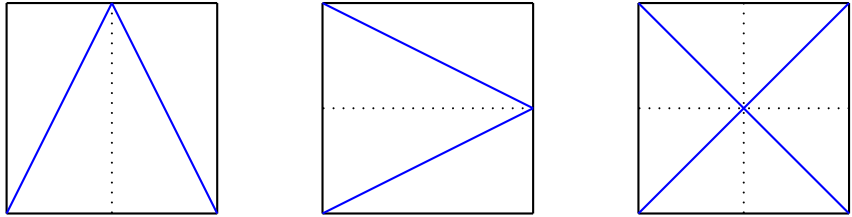 
As mentioned before, the copula $C*C^T$, though neither left nor right invertible, has unit $*$-norm. 
\section{An application: a new measure of dependence}

In \cite{renyi1959}, R\'{e}nyi triggered numerous interests in finding the ``right''  sets of properties that a natural (if any) measure of dependence $\delta(X,Y)$ should possess.  
For reference, the seven postulates 
proposed by R\'{e}nyi are listed below. 
\begin{enumerate}
\renewcommand{\theenumi}{\alph{enumi}}
  \item \label{Renyi:def} $\delta(X,Y)$  is defined for all non-constant random variables $X$, $Y$.
  \item \label{Renyi:sym} $\delta(X,Y) = \delta(Y,X)$.
  \item \label{Renyi:01} $\delta(X,Y) \in [0,1]$.
  \item \label{Renyi:0} $\delta(X,Y) = 0$ if and only if $X$ and $Y$ are independent.
  \item \label{Renyi:1} $\delta(X,Y) = 1$ if either $Y=f(X)$ or $X=g(Y)$ a.s.~for some Borel-measurable functions $f$, $g$.
  \item \label{Renyi:scale} If $\alpha$ and $\beta$ are Borel bijections on $\R$ then $\delta(\alpha(X),\beta(Y)) = \delta(X,Y)$.
  \item \label{Renyi:rho} If $X$ and $Y$ are jointly normal with correlation coefficient $\rho$, then $\delta(X, Y ) = |\rho|$.
\end{enumerate}
Thus far, the only measure of dependence that satisfies all of the above postulates is the maximal correlation coefficient introduced by Gebelein \cite{geb41}.  See for instance \cite{renyi1959,Siburg2009mmc}.

Recently, Siburg and Stoimenov \cite{Siburg2009mmc} introduced a measure of mutual complete dependence $\omega(X,Y)$ defined via its copula $C_{X,Y}$ by $\omega(X,Y)=\sqrt{3}\|C_{X,Y}-\Pi\|$.  While $\omega$ is defined only for continuous random variables, it satisfies the next three properties \ref{Renyi:sym}.--\ref{Renyi:0}.~enjoyed by most if not all measures of dependence.  However, instead of the conditions \ref{Renyi:1}.~and \ref{Renyi:scale}., $\omega$ satisfies the following conditions which makes it suitable for capturing mutual complete dependence regardless of how the random variables are related.
\begin{itemize}
  \item[\ref{Renyi:1}.$'$] $\omega(X,Y) = 1$ if and only if there exist Borel measurable bijections $f$ and $g$ such that $Y=f(X)$ and $X=g(Y)$ almost surely.
  \item[\ref{Renyi:scale}.$'$] If $\alpha$ and $\beta$ are strictly monotonic transformations on images of $X$ and $Y$, respectively, then $\omega(\alpha(X),\beta(Y)) = \omega(X,Y)$.
\end{itemize}
Now, the property \ref{Renyi:scale}.$'$ means that $\omega$ is invariant under only strictly monotonic transformations of random variables.  
\commentout{
The following theorem summarizes properties of the measure.
\begin{thm}[\cite{Si9}, Theorem 5.3] \label{omega}
 \draftnote{Erase this Thm}
Let $X$ and $Y$ be continuous random variables with copula $C$. Then the measure $\omega(X,Y)$ has the following properties:
\begin{enumerate}
\item $\omega(X,Y)= \omega(Y,X)$.
\item $0 \le \omega(X,Y) \le 1.$
\item $\omega(X,Y)= 0$ if and only if $X$ and $Y$ are independent.
\item $\omega(X,Y) = 1$ if $X$ and $Y$ are completely dependent.
\item $\omega(X,Y) \in (\sqrt{2}/2,1]$ if $Y$ is completely dependent on $X$ or $X$ is completely dependent on $Y$.
\item If $f$ and $g$ are monotone transformations, then $\omega(f(X),g(Y))= \omega(X,Y).$
\item If $\{(X_n,Y_n)\}_{n \in \mathbb{N}}$ is a sequence of pairs of continuous random variables with copulas $\{C_n\}_{n \in \mathbb{N}}$ and if $\displaystyle\lim\limits_{n \rightarrow \infty}\|C_n-C\|=0$, then we have $\displaystyle\lim\limits_{n \rightarrow \infty}\omega(X_n,Y_n)=\omega(X,Y).$
\end{enumerate}
\end{thm}
}
Using the $*$-norm which is invariant under all Borel measurable bijections, we define $$\omega_*(X,Y)=\sqrt{3}\|C_{X,Y}-\Pi\|_* =(3\|C_{X,Y}\|_*^2-2)^{1/2},$$ where the last equality follows from Proposition \ref{props}(3). 
Since the $*$-norm shares many properties with the Sobolev norm (see Proposition \ref{prop:*norm}), the properties of $\omega_*$ are for the most part analogous to those of $\omega$'s.  Main exceptions are that \ref{Renyi:1}.$'$--\ref{Renyi:scale}.$'$ are replaced back by \ref{Renyi:1}.--\ref{Renyi:scale}. 
%
\begin{thm}
Let $X$ and $Y$ be continuous random variables with copula $C$. Then $\omega_*(X,Y)$ has the following properties:
\begin{enumerate}
\item $\omega_*(X,Y)= \omega_*(Y,X)$.
\item $0 \le \omega_*(X,Y) \le 1.$
\item $\omega_*(X,Y)= 0$ if and only if $X$ and $Y$ are independent.
\item \label{*:1} $\omega_*(X,Y) = 1$ if  $Y$ is completely dependent on $X$ or $X$ is completely dependent on $Y$.
\item \label{*:scale} If $f$ and $g$ are Borel measurable bijective transformations, then we have $\omega_*(f(X),g(Y))= \omega_*(X,Y).$
\item If $\{(X_n,Y_n)\}_{n \in \mathbb{N}}$ is a sequence of pairs of continuous random variables with copulas $\{C_n\}_{n \in \mathbb{N}}$ and if $\displaystyle\lim\limits_{n \rightarrow \infty}\|C_n-C\|_*=0$, then $\displaystyle\lim\limits_{n \rightarrow \infty}\omega_*(X_n,Y_n)=\omega_*(X,Y).$
\end{enumerate}
\end{thm}
\begin{proof} 
1.~follows from the fact that $\|C_{X,Y}\|_*=\|C_{Y,X}\|_*$.  See Proposition \ref{props}. 
2.~is clear from the definition of $\|\cdot\|_*$ and the fact that $\|C_{X,Y}\|^2 \in [2/3,1]$. 
The statement 3.~is a result of Proposition \ref{props} which says that $\|C_{X,Y}\|_*^2=2/3$ if and only if $C_{X,Y}=\Pi$. 
4.~follows immediately from Theorem \ref{thm:cdnorm1}. 
To prove 5., let $f,g$ be Borel measurable bijective transformations. Then, $X$ and $f(X)$ are mutually completely dependent, and so are $Y$ and $g(Y)$. Thus $\|C_{f(X),X}\|=1$ and $\|C_{Y,g(Y)}\|=1$. Therefore, the copulas $C_{f(X),X}$ and $C_{Y,g(Y)}$ are invertible. 
Hence 
$\omega_*(f(X),g(Y)) 
= \sqrt{3}\|C_{f(X),X}*(C_{X,Y}-\Pi)*C_{Y,g(Y)}\|_*
= \sqrt{3}\|C_{X,Y}-\Pi \|_* = \omega_*(X,Y).
$  
  Finally, 6.~can be proved via the inequality 
\[ |\omega_*(X_n,Y_n)-\omega_*(X,Y)| = \sqrt{3}\abs{\|C_n-\Pi\|_* - \|C-\Pi\|_*}
\le \sqrt{3}\|C_n-C\|_*.\]
\end{proof}

Therefore, we have constructed a measure of dependence for continuous random variables which satisfies all of Renyi's postulates except possibly the last condition \ref{Renyi:rho}.  The $*$-norm of a convex sum of a unit $*$-norm copula and the independence copula is computed.
\begin{exam} 
By the computations in Example \ref{exam:EqualNorms}, if $\|A\|_* = 1$ and  $C_{X,Y} = \alpha A + (1-\alpha)\Pi$, then $\omega_*(X,Y) = [3(\alpha^2+2)/3-2]^{1/2} = \alpha.$
\end{exam}
Corollary \ref{cor:L*R} implies that there are many more copulas with unit $*$-norm, i.e.~any copulas of the form $C_{X,f(X)}*C_{g(Y),Y}$ where $f,g$ are Borel measurable transformations.  By the characterization of idempotent copulas in \cite{darsow2010}, all singular idempotent copulas are of this form and hence have unit $*$-norm.

\end{document}